\documentclass[a4paper,11pt]{article}

\usepackage[utf8x]{inputenc}
\usepackage{amssymb}
\usepackage{amsmath}
\usepackage[english]{babel}
\usepackage[colorlinks, linkcolor=red, urlcolor=blue, citecolor=blue]{hyperref}

\newtheorem{coro}{Corollary}[section]
\newtheorem{prop}{Proposition}[section]
\newtheorem{defn}{Definition}[section]
\newtheorem{exple}{Example}[section]
\newtheorem{rem}{Remark}[section]
\newenvironment{proof}[1][Proof]{\textbf{#1.}}{\hfill $\Box$}
\begin{document}

\date{}
\title{\LARGE{A note on $\alpha$-IDT processes}\footnote{This work is supported by Hassan II Academy of Sciences and Technologies}}
\author{
\small{Antoine HAKASSOU}\footnote{\textbf{a.hakassou@edu.uca.ma}, Cadi Ayyad University, LIBMA Laboratory, Department of Mathematics,
Faculty of Sciences Semlalia P.B.O. 2390 Marrakesh, Morocco.}
\, and 
\small{Youssef OUKNINE}\footnote{\textbf{ouknine@uca.ma}, Cadi Ayyad University of Marrakesh and Hassan II Academy of Sciences and Technologies Rabat.}
}

\maketitle

\begin{center}
\textbf{Abstract}
\end{center}

In this note, we introduce the notion of $\alpha$-IDT processes which is obtained from a slight and fundamental modification of the IDT property. 
Several examples of $\alpha$-IDT processes are given and Gaussian processes which are $\alpha$-IDT are characterized. A kind example of this Gaussian 
$\alpha$-IDT is the standard fractional Brownian motion. Also, we invest some links between the $\alpha$-IDT property, with
selfdecomposability, temporal selfdecomposability, stability and self similarity.\\


\section{Introduction}

This introductory section is devoted to give the definition of $\alpha$-IDT processes and to exhibit numerous examples of such processes. 
This notion of $\alpha$-IDT is obtained from a slight and fundametal modification of the IDT property introduced by R. Mansuy \cite{Mansuy}. 
This class of processes turn out to be also rich than the class of IDT processes since we shall prove that fractional brownian motion with 
any Hurst index $H\in(0;1)$, are $\alpha$-IDT. This is not true in general for IDT processes.

\begin{defn}
 
 Let $\alpha>0$ and $X=(X_{t};t\geq 0)$ a stochastic process. X is said to be an $\alpha$-IDT process if, and only if, 
 for any $n\in\mathbb{N}^{\ast}$ we have
 \begin{equation} \label{alphaIDT}
  (X_{n^{1/\alpha}t}; t\geq 0) {\stackrel{(law)}{=}}(X_{t}^{(1)}+\cdots+X_{t}^{(n)}; t\geq 0)
 \end{equation}
where $(X^{(i)})_{1\leq i\leq n}$ are independent copies of X.
 
\end{defn}

\begin{rem}

 For $\alpha=1$, this definition is reduced to the well-known class of IDT processes (see R. Mansuy \cite{Mansuy}, 
 K. Es-Sebaiy $\&$ Y. Ouknine \cite{Ouknine}, and A. Hakassou $\&$ Y. Ouknine \cite{Ouknine2}). 
 
\end{rem}

Let us now, provide many examples as possible of $\alpha$-IDT processes which are not necessarily IDT processes.

\begin{exple}
 
 Let $0<\alpha\leq2$ and consider a strictly stable random variable $S_{\alpha}$ with parameter $\alpha$. 
 Then, the stochastic process X defined by:
 $$X_{t}=tS_{\alpha}, \, t\in\mathbb{R}_{+},$$
 is an $\alpha$-IDT process. 
 
\end{exple}

\begin{exple}
 
 Let $\alpha>0$ and consider a strictly 1-stable random variable S. Then, the stochastic process X defined by 
 $$X_{t}=t^{\alpha}S, \, t\in\mathbb{R}_{+},$$
 is an $\alpha$-IDT process.
 
\end{exple}

\begin{exple}
 
 If X is an $\alpha$-IDT process and $\mu$ a measure on $\mathbb{R}_{+}$ such that 
 $$X^{(\mu)}_{t}=\int \mu(du)X_{ut}, \,t\in\mathbb{R}_{+}$$
 is well defined, then $X^{(\mu)}$ is again an $\alpha$-IDT process.
 
\end{exple}

\section{Gaussian $\alpha$-IDT processes}

The next proposition emphasizes again how much the class $\alpha$-IDT processes is rich, since we shall prove that fractional 
Brownian motion of any order $H\in (0,1)$ are $\alpha$-IDT with $\alpha=2H$.

\begin{prop}
 
 Let $\alpha>0$ and $(G_{t};t\geq0)$ be a centered Gaussian process which is assumed to be continuous in probability. 
 Then the following properties are equivalent:
\begin{enumerate}
\item  $(G_{t}; t\geq 0)$ is an $\alpha$-IDT process.
\item  The covariance function $c(s,t):=\mathbb{E}[G_{s}G_{t}]$, $0 \leq s \leq t,$ satisfies
                   $$\forall  a> 0,  c(as, at)=a^{\alpha} c(s,t), \mbox{ for all   }  0 \leq s \leq t.$$
\item  The process $(G_{t}; t\geq 0)$ satisfies the scaling property of order $h=\frac{\alpha}{2}$, namely
          $$\forall a> 0, \mbox{  } (G_{at}; t \geq 0)\stackrel{(law)}{=}(a^{\alpha/2} G_{t}; t\geq 0).$$
\item  The process $(\tilde{G}_{y}:=e^{-\frac{\alpha}{2}y} G_{e^{y}}; y\in\mathbb{R})$ is stationary.
\item  The covariance function $\tilde{c}(y,z):=\mathbb{E}[\tilde{G}_{y} \tilde{G}_{z}],$  $y,z\in\mathbb{R},$ is of
the form $$\tilde{c}(y,z)=\int \mu (du) e^{iu \mid y-z \mid},  \mbox{ y,z}\in\mathbb{R}$$
where $\mu$ is a positive, finite, symmetric measure on $\mathbb{R}$.
\end{enumerate}
Then, under these equivalent conditions, the covariance function c of $(G_{t}; t\geq 0)$ 
is given by $$c(s,t)=(st)^{\alpha/2} \int \mu (da) e^{ia \mid ln(\frac{s}{t}) \mid}.$$

\end{prop}

\begin{proof}

$(1 \Leftrightarrow 2)$ The $\alpha$-IDT property \eqref{alphaIDT} is equivalent to
\begin{displaymath}
\forall n\in\mathbb{N}, \hspace{1cm} c(n^{1/\alpha}s, n^{1/\alpha}t)=nc(s,t), \hspace{1cm} 0\leq s \leq t
\end{displaymath}
and also
\begin{displaymath}
\forall q\in\mathbb{Q}_{+}, \hspace{1cm} c(q^{1/\alpha}s,q^{1/\alpha}t)=qc(s,t), \hspace{1cm} 0\leq s \leq t
\end{displaymath}
The result is then obtained using the density of $\mathbb{Q}_{+}$ in $\mathbb{R}_{+}$ and the continuity of c (deduced by the continuity in probability, 
hence in $L^{2}$, of $(G_{t}, t\geq0)$).\\

$(2 \Leftrightarrow 3)$ Since the law of a centered Gaussian process is determined by its covariance function, the result follows immediately.\\

$(3 \Leftrightarrow 4)$ Apply Lamperti's transformation of order $h=\alpha/2.$ \\

$(4 \Leftrightarrow 5)$ Bochner's theorem for definite positive functions.\\

\end{proof}

\begin{exple}
 
 Consider $B_{H}=(B_{H}(t);t\geq0)$ the normalized fractional Brownian motion with Hurst index $H\in (0,1)$, that is a centered continuous Gaussian
 process such that:
 $$\mathbb{E}[B_{H}(t)B_{H}(s)]=\frac{1}{2}(|t|^{2H}+|s|^{2H}-|t-s|^{2H}), \, 0\leq s\leq t.$$
 Then $B_{H}=(B_{H}(t);t\geq0)$ is an $2H$-IDT Gaussian process. 
 
\end{exple}

\begin{exple}
 
 Let $\varphi\in\L^{2}(\mathbb{R}_{+};du)$; Then the process $G^{\varphi}$ defined by 
 $$\forall t>0, \, G^{\varphi}_{t}=\int_{0}^{\infty} \varphi (\frac{u}{t})dB_{H}(u), \, \mbox{ and } G^{\varphi}_{0}=0,$$
 where $B_{H}$ is the normalized fractional Brownian motion with Hurst index $H\in (0,1)$, is an $2H$-IDT Gaussian process.
 
\end{exple}

\section{$\alpha$-IDT and additive processes with the same marginals}

Let us first note that the laws of finite-dimensional marginals of an $\alpha$-IDT process are infinitely divisible. In particular, for any fixed t, the 
law of $X_{t}$ is infinitely divisible.

\begin{prop}\label{asso}
 
 Let $(X_{t};t\geq0)$ be a stochastically continuous $\alpha$-IDT process. Denote by $(L_{t};t\geq0)$ the unique Lévy process such that 
 $$X_{1}\stackrel{law}{=}L_{1}.$$
 Then $(X_{t};t\geq0)$ and the additive process $(L_{t^{\alpha}};t\geq0)$ have the same one-dimensional marginals, i.e. for any fixed $t\geq0$, 
 \begin{equation}\label{association}
  X_{t}\stackrel{law}{=}L_{t^{\alpha}}
 \end{equation}

\end{prop}

\begin{proof}
 
 From the definition of $\alpha$-IDT process, it easy to check that for any $n\in\mathbb{N}^{\ast}$, 
 $$\mathbb{E}e^{i\theta X_{n^{1/\alpha}}}=(\mathbb{E}e^{i\theta X_{1}})^{n} \mbox{ for all } \theta\in\mathbb{R}$$
 and also for any rational $q\in\mathbb{Q}_{+}$ that 
 $$\mathbb{E}e^{i\theta X_{q^{1/\alpha}}}=(\mathbb{E}e^{i\theta X_{1}})^{q} \mbox{ for all } \theta\in\mathbb{R}.$$
 Now, thanks to the stochastic continuity of X and the density of $\mathbb{Q}_{+}$ in $\mathbb{R}_{+}$, it follows that 
 for any $t\geq0$, 
 $$\mathbb{E}e^{i\theta X_{t^{1/\alpha}}}=(\mathbb{E}e^{i\theta X_{1}})^{t} \mbox{ for all } \theta\in\mathbb{R}.$$
 Since $X_{1}\stackrel{law}{=}L_{1}$ and thanks to the stochastic continuity of X and L, we get for any $t\geq0$,
 $$\mathbb{E}e^{i\theta X_{t^{1/\alpha}}}=\mathbb{E}e^{i\theta L_{t}} \mbox{ for all } \theta\in\mathbb{R}.$$
 The proof is completed.
 
 \end{proof}
 
 \begin{coro}
  
  If $(X_{t};t\geq0)$ is an $\alpha$-IDT process, stochastically continuous with independent increments, then it is a time-changed Lévy process with the 
  deterministic chronometer $h(t)=t^{\alpha}$.
  
 \end{coro}

 \begin{proof}
 
 The proof is consequence of Proposition \ref{asso} and Theorem 9.7 in K. Sato \cite{Sato}.
 
 \end{proof}

\section{A link with path-valued Lévy processes}

From the definition of an $\alpha$-IDT process, it follows that viewed as a random variable taking values in path-space $D=D([0,\infty))$, i.e the 
classical Skorohod space of right-continuous paths over $[0,\infty)$, this random variable is infinitely divisible. We note that there exist D-valued
infinitely divisible variables which are $\alpha$-IDT. The main task of this section is to characterize D-valued infinitely divisible variable which are 
$\alpha$-IDT.

\begin{prop} \label{lemma}
 
 Let X be a D-valued infinitely divisible variable. X is an $\alpha$-IDT process if, and only if, 
 \begin{itemize}
  \item its Lévy measure M over D satisfies for any non-negative functional F on D,
      \begin{equation} \label{path1}
        \int_{D} M(dy)F(y(n\cdot))=n^{\alpha}\int_{D} M(dy)F(y(\cdot))
      \end{equation}
  \item its Gaussian measure $\rho$ over D satisfies for any $v,u\in\mathbb{R}$, 
	\begin{equation}\label{path2}
        \int_{D} y(nu)y(nv)\rho(dy)=n^{\alpha}\int_{D} y(u)y(v)\rho(dy)
        \end{equation}
 \end{itemize}
 
\end{prop}

\begin{proof}
 
 The $\alpha$-IDT property \eqref{alphaIDT} admits the following equivalent formulation: for any $f\in C_{c}(\mathbb{R}_{+}, \mathbb{R}_{+})$, 
 \begin{equation*}
  \mathbb{E}[exp -\int_{0}^{\infty} dt f(t)X_{n^{1/\alpha}t}]=(\mathbb{E}[exp -\int_{0}^{\infty} dt f(t)X_{t}])^{n}.
 \end{equation*}
That is 
     \begin{equation*}
        \int_{D} M(dy)(1-e^{-\int_{0}^{\infty} dt f(t)y(nt)})=n^{\alpha}\int_{D} M(dy)(1-e^{-\int_{0}^{\infty} dt f(t)y(t)})
      \end{equation*}
 and 
\begin{equation*}
\int_{D} \int_{0}^{\infty} du f(u)y(nu)\int_{0}^{\infty} dt f(t)y(nt)\rho(dy)=
  n^{\alpha}\int_{D} \int_{0}^{\infty} du f(u)y(u)\int_{0}^{\infty} dt f(t)y(t)\rho(dy)
\end{equation*}
from which \eqref{path1} and \eqref{path2} follow easily.

\end{proof}

Let us now provide some constructions of $\alpha$-IDT processes.

\begin{prop}

 Let N be a Lévy measure on path-space D. Define M as follows 
 \begin{equation}
   \int_{D} M(dy)F(y(\cdot))=\int_{0}^{\infty}du \int_{D} N(dy)F(y(\frac{\cdot}{u^{1/\alpha}}))
 \end{equation}
Then M is an $\alpha$-IDT Lévy measure, i.e. the Lévy measure of an $\alpha$-IDT process viewed as a D-valued infinitely divisible random variable. 

\end{prop}

\begin{proof}
 
Thanks to the trivial change of variable $u=n^{\alpha}v$, we get 
\begin{equation*}
\begin{split}
  \int_{D} M(dy)F(y(n\cdot))&=\int_{0}^{\infty}du \int_{D} N(dy)F(y(\frac{n\cdot}{u^{1/\alpha}}))\\
  &=n^{\alpha}\int_{0}^{\infty}dv \int_{D} N(dy)F(y(\frac{\cdot}{v^{1/\alpha}}))\\
  &=n^{\alpha}\int_{D} M(dy)F(y(\cdot))
\end{split}
\end{equation*}
And the reasult is now obvious by Proposition \ref{lemma}. 
 
\end{proof}

\begin{exple}

 Let X be a subordinator without drift, $\nu$ its Lévy measure, $\varphi$ a regular function and define 
 \begin{equation*}
  X_{t}^{(\varphi)}=\int_{0}^{\infty} du \varphi(u)X_{(ut)^{\alpha}}
 \end{equation*}
Then $(X^{(\varphi)}_{t};t\geq0)$ is an $\alpha$-IDT process and its $\alpha$-IDT Lévy measure satisfies for any functional over D,
\begin{equation*}
 \int_{D} M^{(\varphi)}(dy)F(y(\cdot))=\int_{0}^{\infty} du \int _{D} \nu(dx)F(x\Phi(\frac{u^{1/\alpha}}{\cdot}))
\end{equation*}
where $\Phi$ is the tail of the integral of $\varphi$: $\Phi(u)=\int_{u}^{\infty}dv\varphi(v)$.

\end{exple}

\section{Link with selfsimilarity and stability}

\begin{prop}

Let $(X_{t};t\geq0)$ be a non-trivial stochastically continuous $\alpha$-IDT process. Then $(X_{t};t\geq0)$ is strictly $\beta$-stable if and only if it is $(\alpha/\beta)$-selfsimilar. 

\end{prop}

\begin{proof}

First, assume that $(X_{t};t\geq0)$ is $(\alpha/\beta)$-selfsimilar. Then for all $n\in\mathbb{N}^{\ast}$ we have 
$$(X_{n^{1/\alpha}t};t\geq0)\stackrel{(law)}{=}(n^{1/\beta}X_{t};t\geq0)$$
and using the $\alpha$-IDT property, we obtain
$$
(n^{1/\beta}X_{t};t\geq0)\stackrel{(law)}{=}(X_{n^{1/\alpha}t};t\geq0)\stackrel{(law)}{=}
(\sum_{j=1}^{j=n}X^{(j)}_{t};t\geq0)
$$ 
where $X^{(1)},\cdots, X^{(n)}$ are independent copies of X. Thus X is strictly $\beta$-stable. \par
Conversely, suppose that $(X_{t};t\geq0)$ is strictly $\beta$-stable. Since $(X_{t};t\geq0)$ is $\alpha$-IDT, it follows that for any $n\in\mathbb{N}^{\ast}$
$$(X_{n^{1/\alpha}t};t\geq0)\stackrel{(law)}{=}(()n^{1/\alpha})^{\alpha/\beta}X_{t};t\geq0)$$
and also for any rational $q\in\mathbb{Q}_{+}$,
$$(X_{q^{1/\alpha}t};t\geq0)\stackrel{(law)}{=}((q^{1/\alpha})^{\alpha/\beta}X_{t};t\geq0).$$
Thanks the stochastic continuity, it follows that X is $(\alpha/\beta)$-selfsimilar. 

\end{proof}

\begin{prop}

A non-trivial strictly $\beta$-stable and $(\alpha/\beta)$-selfsimilar process $(X_{t};t\geq0)$ is an $\alpha$-IDT process.

\end{prop}

\begin{proof}

Since X is strictly $\beta$-stable, we have for all $n\in\mathbb{N}^{\ast}$
$$(n^{1/\beta}X_{t};t\geq0)\stackrel{(law)}{=}(\sum_{j=1}^{j=n}X^{(j)}_{t};t\geq0)$$
where $X^{(1)}, \cdots, X^{(n)}$ are independent copies of X. On the other hand, it follows from the $(\alpha/\beta)$-selfsimilarity of X that 
$$(X_{n^{1/\alpha}t};t\geq0)\stackrel{(law)}{=}(n^{1/\beta}X_{t};t\geq0)$$
which implies that $(X_{t};t\geq0)$ is an $\alpha$-IDT process.

\end{proof}

\section{Subordination through $\alpha$-IDT process}

\begin{prop}

Let $(X_{t};t\geq0)$ be a Lévy process and $(\xi_{t};t\geq0)$ an $\alpha$-IDT chronometer such that X and $\xi$ are independent. Then, $(Z_{t}:=X_{\xi_{t}};t\geq0)$ is an $\alpha$-IDT process. 

\end{prop}

\begin{proof}

Let $\xi^{(l)}$, $l=1,\cdots,n$, be independent copies of $\xi$. Since X is independent of $\xi$, then for every $m\geq1$ and
$\theta=(\theta_{1},\cdots,\theta_{m})\in\mathbb{R}^{m},$ we have
\begin{equation*}
\begin{split}
 J(n,\theta):&=\mathbb{E}exp\{i\sum_{k=1}^{m} <\theta_{k}, X(\xi_{n^{1/\alpha}t_{k}})>\}\\
&=\mathbb{E}[(\mathbb{E}exp \{i\sum_{k=1}^{m} <\theta_{k}, X(s_{k})>\})_{s_{k}=\xi_{n^{1/\alpha}t_{k}},\, k=1,\cdots,m}].
\end{split}
\end{equation*}
Using the $\alpha$-IDT property, we obtain
$$J(n,\theta)=\mathbb{E}\big[(\mathbb{E}exp \{i\sum_{k=1}^{m}
<\theta_{k}, X(s_{k})>\})_{s_{k}=\sum_{l=1}^{n}\xi^{(l)}_{t_{k}},\, k=1,\cdots,m}\big].$$
According to the change of variables $\lambda_{k}=\theta_{k} + \cdots + \theta_{m}$ and $t_{0}=0$,  we have
$$J(n,\theta)=\mathbb{E}[(\mathbb{E}exp \{i\sum_{k=1}^{m} <\lambda_{k}, X(s_{k})-X(s_{k-1})>\})_{s_{k}
=\sum_{l=1}^{n}\xi^{(l)}_{t_k},\, k=1,\cdots,m}].$$
By the independence of increments of X, we get
$$J(n,\theta)=\mathbb{E}[(\prod_{k=1}^{m} \mathbb{E}exp
\{i <\lambda_{k}, X(s_{k})-X(s_{k-1})>\})_{s_{k}=\sum_{l=1}^{n}\xi^{(l)}_{t_{k}}, \, k=1,\cdots,m}].$$
Now, it follows from the stationary of the increments of X and the independence of the $\xi^{(l)}$, $l=1,\cdots,n$, that
$$J(n,\theta)=\mathbb{E}[(\prod_{k=1}^{m}\prod_{l=1}^{n} \mathbb{E}exp
\{i <\lambda_{k}, X(r_{l,k})>\})_{r_{l,k}=\xi^{(l)}_{t_{k}}-\xi^{(l)}_{t_{k-1}},\,
k=1,\cdots,m;\, l=1,\cdots,n}].$$
And that is
$$J(n,\theta)=\mathbb{E}[(\prod_{l=1}^{n} \mathbb{E}exp \{i \sum_{k=1}^{m}
<\theta_{k}, X(r_{l,k})>\})_{r_{l,k}=\xi^{(l)}_{t_{k}},\, k=1,\cdots,m;\, l=1,\cdots,n}].$$
Now, this implies
$$J(n,\theta)=(\mathbb{E}[(\mathbb{E}exp \{i \sum_{k=1}^{m} <\theta_{k}, X(\xi_{r_{k}})>\})_{r_{k}
=\xi_{t_{k}},\, k=1,\cdots,m}])^{n}.$$
Hence
$$J(n,\theta)=(\mathbb{E}exp \{i \sum_{k=1}^{m} <\theta_{k}, X(\xi_{t_{k}})>\})^{n}.$$
The proof is achieved.

\end{proof}

\section{Links with temporal selfdecomposability and selfdecomposability}

\begin{prop}

A stochastically continuous $\alpha$-IDT process is temporally selfdecomposable of infinite order.

\end{prop}

\begin{proof}

For any $(\theta_{1},\cdots,\theta_{m})\in\mathbb{R}^{m}$, it follows thanks to the $\alpha$-IDT property that for any $n\in\mathbb{N}^{\ast}$,
\begin{equation*}
 \mathbb{E}exp\{i\sum_{k=1}^{m}<\theta_{k},X_{t_{k}}>\}=(\mathbb{E}exp\{i\sum_{k=1}^{m}<\theta_{k},X_{n^{1/\alpha}t_{k}}>\})^{1/n}
\end{equation*}
And for any $q\in\mathbb{Q}_{+}^{\ast}$,
\begin{equation*}
 \mathbb{E}exp\{i\sum_{k=1}^{m}<\theta_{k},X_{t_{k}}>\}=(\mathbb{E}exp\{i\sum_{k=1}^{m}<\theta_{k},X_{q^{1/\alpha}t_{k}}>\})^{1/q}
\end{equation*}
In particular for $b\in\mathbb{Q}_{+}^{\ast}\cup(0;1)$, we get 
\begin{equation*}
\begin{split}
 \mathbb{E}exp\{i\sum_{k=1}^{m}<\theta_{k},X_{t_{k}}>\}=(\mathbb{E}exp\{i\sum_{k=1}^{m}<\theta_{k},X_{b^{1/\alpha}t_{k}}>\})\\
 \times(\mathbb{E}exp\{i\sum_{k=1}^{m}<\theta_{k},X_{b^{1/\alpha}t_{k}}>\})^{\frac{1}{b}-1}
 \end{split}
\end{equation*}
That is 
\begin{equation*}
\begin{split}
\mathbb{E}exp\{i\sum_{k=1}^{m}<\theta_{k},X_{t_{k}}>\}=(\mathbb{E}exp\{i\sum_{k=1}^{m}<\theta_{k},X_{b^{1/\alpha}t_{k}}>\})\\
 \times(\mathbb{E}exp\{i\sum_{k=1}^{m}<\theta_{k},X_{(\frac{b}{b'})^{1/\alpha}t_{k}}>\}) \mbox{ with } \frac{1}{b'}=\frac{1}{b}-1.
 \end{split}
\end{equation*}
By stochastic continuity, it follows that for any $b\in(0;1)$, we get 
\begin{equation*}
\begin{split}
\mathbb{E}exp\{i\sum_{k=1}^{m}<\theta_{k},X_{t_{k}}>\}=(\mathbb{E}exp\{i\sum_{k=1}^{m}<\theta_{k},X_{b^{1/\alpha}t_{k}}>\})\\
 \times(\mathbb{E}exp\{i\sum_{k=1}^{m}<\theta_{k},X_{(\frac{b}{b'})^{1/\alpha}t_{k}}>\})
 \end{split}
\end{equation*}
Now we set $c=b^{1/\alpha}$ and $c'=(b')^{1/\alpha}$, and then we deduce that $(X_{t};t\geq0)$ is temporally selfdecomposable and the c-residual has finite 
dimensional marginals which fit with those of the process X suitably rescaled, the result follows.

\end{proof}

\begin{prop}

Let $(X_{t};t\geq0)$ be an $\alpha$-IDT process. If X is selfdecomposable, then its residual is also an $\alpha$-IDT 
process.

\end{prop}

\begin{proof}

By definition (see O. E. Barndorff-Nielsen \it{et al} \cite{K.Sato}), if X is selfdecomposable then for any $c\in (0;1)$ there exists a process 
$U^{(c)}=(U^{(c)}_{t};t\geq0)$ called the c-residual of X, such that 
\begin{equation*}
 X\stackrel{(law)}{=}cX'+U^{(c)}
\end{equation*}
where X' is an independent copie of X and X' and $U^{(c)}$ are independent. \\
Thus is follows that for any $(\theta_{1},\cdots,\theta_{m})$ we have
\begin{equation*}
 \mathbb{E}exp \{i\sum_{k=1}^{m}<\theta_{k},U^{(c)}_{t_{k}}>\}=\frac{\mathbb{E}exp \{i\sum_{k=1}^{m}<\theta_{k},X_{t_{k}}>\}}
 {\mathbb{E}exp \{i\sum_{k=1}^{m}<\theta_{k},cX'_{t_{k}}>\}}
\end{equation*}
Now, for any $n\in\mathbb{N}^{\ast}$, we have thanks to the $\alpha$-IDT property of X and X':
\begin{equation*}
 \mathbb{E}exp \{i\sum_{k=1}^{m}<\theta_{k},U^{(c)}_{n^{1/\alpha}t_{k}}>\}=\frac{(\mathbb{E}exp \{i\sum_{k=1}^{m}<\theta_{k},X_{t_{k}}>\})^{n}}
 {(\mathbb{E}exp \{i\sum_{k=1}^{m}<\theta_{k},cX'_{t_{k}}>\})^{n}}
\end{equation*}
And that is $U^{(c)}$ is an $\alpha$-IDT process.

\end{proof}

\end{document}